\documentclass{mrlart7}
  \def\volno{??}
  \def\yrno{2017}
  \def\issueno{00}
  \def\lpageno{10015} 
\usepackage{amsmath}
\usepackage{amssymb}
\usepackage{amsfonts}
\usepackage{amsthm}
\usepackage{mathrsfs}
\usepackage[all]{xy}
\usepackage{color}

\makeatletter
\def\@tocline#1#2#3#4#5#6#7{\relax
  \ifnum #1>\c@tocdepth 
  \else
    \par \addpenalty\@secpenalty\addvspace{#2}%
    \begingroup \hyphenpenalty\@M
    \@ifempty{#4}{%
      \@tempdima\csname r@tocindent\number#1\endcsname\relax
    }{%
      \@tempdima#4\relax
    }%
    \parindent\z@ \leftskip#3\relax \advance\leftskip\@tempdima\relax
    \rightskip\@pnumwidth plus4em \parfillskip-\@pnumwidth
    #5\leavevmode\hskip-\@tempdima
      \ifcase #1
       \or\or \hskip 1em \or \hskip 2em \else \hskip 3em \fi%
      #6\nobreak\relax
    \dotfill\hbox to\@pnumwidth{\@tocpagenum{#7}}\par
    \nobreak
    \endgroup
  \fi}
\makeatother

\usepackage[colorlinks=true, pdfstartview=FitV, linkcolor=blue, citecolor=blue, urlcolor=blue]{hyperref}
\newcommand{\arxiv}[1]{\href{http://arxiv.org/abs/#1}{\tt arXiv:\nolinkurl{#1}}}

\theoremstyle{plain}
\newtheorem*{theorem*}{Theorem}
\newtheorem{theorem}{Theorem}[section]
\newtheorem{lemma}[theorem]{Lemma}

\newtheorem{proposition}[theorem]{Proposition}

\newtheorem{lem}[theorem]{Lemma}

\newtheorem{prop}[theorem]{Proposition}

\theoremstyle{definition}
\newtheorem{definition}[theorem]{Definition}
\newtheorem{remark}[theorem]{Remark}
\newtheorem{defn}[theorem]{Definition}
\newtheorem{example}[theorem]{Example}

\renewcommand{\Re}{{\rm Re}\,}

\newcommand{\R}{\mathbb{ R}}
\newcommand{\C}{\mathbb{ C}}

\newcommand{\F}{\mathbb{F}}

\DeclareMathOperator{\Ext}{Ext}

\DeclareMathOperator{\Hom}{Hom}

\DeclareMathOperator{\GL}{GL}

\DeclareMathOperator{\im}{im}

\DeclareMathOperator{\Id}{Id}

\DeclareMathOperator{\Surj}{Surj}
\DeclareMathOperator{\Irr}{Irr}
\newcommand{\g}{\mathfrak{g}}

\newcommand{\om}{{\overline{m}}}
\newcommand{\bfv}{{\bf v}}

\newcommand{\wt}{\text{wt}}
\newcommand{\eps}{\varepsilon}
\newcommand{\f}{f}
\newcommand{\e}{e}

\newcommand{\M}{{\text{M}}}

\begin{document}
\bibliographystyle{amsalpha}

\title[Crystals and modulated graphs]{Quiver varieties and crystals in symmetrizable type via modulated graphs}

\author{Vinoth Nandakumar}
\address{School of Mathematics and Statistics \\  University of Sydney  \\ NSW 2006
Australia}
\email{vinoth.90@gmail.com}

\author{Peter Tingley}
\address{Dept. of Math and Stats \\ Loyola University Chicago \\ 1032 W. Sheridan Rd., Chicago, IL 60660, USA}
\email{ptingley@luc.edu}


\begin{abstract} 
Kashiwara and Saito have a geometric construction of the infinity crystal for any symmetric Kac-Moody algebra. The underlying set consists of the irreducible components of Lusztig's quiver varieties, which are varieties of nilpotent representations of a pre-projective algebra. We generalize this to symmetrizable Kac-Moody algebras by replacing Lusztig's preprojective algebra with a more general one due to Dlab and Ringel. In non-symmetric types we are forced to work over non-algebraically-closed fields.
\end{abstract}

\maketitle

\tableofcontents

\section{Introduction}
Fix a symmetrizable Kac-Moody algebra $\mathfrak{g}$. Kashiwara's crystal $B(-\infty)$ is a combinatorial object that encodes a lot of information about $\mathfrak{g}$ and its integrable lowest weight representations. It is usually defined using the corresponding quantized universal enveloping algebra, but it can be realized in other ways. In symmetric type, Kashiwara and Saito \cite{KS97} developed a very useful geometric realization. There the underlying set consists of the irreducible components of Lusztig's nilpotent varieties from \cite[\S12]{L91}: the varieties of nilpotent representations of Lusztig's preprojective algebra. This preprojective algebra is only defined in symmetric types, which is why Kashiwara and Saito work in that generality. 

However, even before Lusztig's work, Dlab and Ringel \cite{DR80} defined the preprojective algebra of a ``modulated graph." There is a natural way to associate a symmetrizable Cartan matrix to any modulated graph, and all symmetrizable Cartan matrices arise this way. As discussed in \cite{Ringel:1998}, Lusztig's preprojective algebra is a special case of this construction. 

Here we generalize Kashiwara and Saito's work by replacing Lusztig's preprojective algebra with Dlab and Ringel's. This gives realizations of $B(-\infty)$ for any symmetrizable Kac-Moody algebra. 
Kashiwara and Saito's proof largely goes through, although we make some modifications.

The current work can perhaps be generalized: Dlab and Ringel actually allow division rings where we use fields. It would be natural to try to realize $B(-\infty)$ in this even more general setting, but this involves some technicalities we prefer to avoid. 

There is a well known way to study $B(-\infty)$ in symmetrizable types by ``folding" the quiver variety for a larger symmetric type (see \cite{Sav05}). There the crystal for the symmetrizable Kac-Moody algebra is the set of irreducible components of the symmetric type quiver variety that are fixed set-wise by a diagram automorphism. However, we feel it aesthetically important to have a quiver variety in symmetrizable types that is {\it actually a representation variety} for some algebra. This may also simplify some proofs by allowing the symmetric and symmetrizable cases to be handled simultaneously. 

Some recent papers of Geiss, Leclerc and Schr\"oer  \cite{GLS1, GLS2, GLS3} also discuss preprojective algebras in symmetrizable type. They take a different approach (using quivers with relations) and do not consider crystals.

\section{Background}

\subsection{Crystals} \label{ss:crystals}
Let $\g$ be a symmetrizable Kac-Moody algebra  with Cartan matrix $C= (c_{ij})_{i \in I}$, and let $D= \text{diag} \{ d_{i} \}_{i \in I}$ be such that $DC$ is symmetric, with the $d_i$ relatively prime positive integers. Let $P$ be the weight lattice of $\g$, $Q$ the root lattice, and $\{ \alpha_i \}$ the simple roots. 
Let $\langle \cdot, \cdot \rangle$ be the pairing between the root lattice and the co-root lattice defined by $\langle \alpha_i, \check{\alpha}_j \rangle =c_{ji}$, and $(\cdot, \cdot)$ be the symmetric bilinear form on $Q$ defined by $(\alpha_i, \alpha_j)= d_i c_{ij}$.

\begin{defn}\label{def:crystal} (\cite[\S7.2]{Kashiwara:1995}) A {\bf combinatorial crystal} is a set $B$ along with functions $\wt \colon B \to P$, and, for each $i \in I$, $\varepsilon_i, \varphi_i \colon B \to {\mathbb Z}$ and $e_i, f_i: B \rightarrow B \sqcup \{ \emptyset \}$, such that
  \begin{enumerate}
  \item $\varphi_i(b) = \varepsilon_i(b) + \langle \wt(b), \check\alpha_i \rangle$.
  \item $e_i$ increases $\varphi_i$ by 1, decreases $\eps_i$ by 1 and increases $\wt$ by $\alpha_i$.
  \item  \label{def:dcp3}$f_i b = b'$ if and only if $e_i b' = b$.
  \end{enumerate}
We often denote a combinatorial crystal simply by $B$, suppressing the other data.
\end{defn}

\begin{remark} In \cite{Kashiwara:1995}, $\varepsilon_i$ and $\varphi_i$ are allowed to be $-\infty$. We do not need that case. 
\end{remark}

\begin{defn}
\label{def:lwcrystal}
A combinatorial crystal is called {\bf lowest weight} if it has an element $b_-$ (the lowest weight element) such that 
\begin{enumerate}
\item $b_-$ can be reached from any $b \in B$ by applying a finite sequence of $f_i$. 
\item \label{eq:pta} For all $b \in B$ and all $i \in I$, $\varphi_i(b) = \max \{ n : f_i^n(b) \neq \emptyset \}$.
\end{enumerate}
  \end{defn}
\noindent For lowest weight combinatorial crystals, $\wt(b_-)$ and the $e_i$ determine the rest of the data. 

Here we are concerned with the infinity crystal $B(-\infty)$, which can be thought of as the crystal for $U^+_q(\g)$. We use the following essentially as our definition. It is a rewording of \cite[Proposition 3.2.3]{KS97}, and can be found in \cite[Proposition 1.4]{TW}.  
\begin{prop} \label{prop:comb-characterizaton2}
Fix a set $B$ and functions $e_i,f_i, e_i^*,f_i^*: B \rightarrow B \cup \{ \emptyset\}$. Assume $(B, e_i, f_i)$ and $(B, e_i^*, f_i^*)$ are both lowest weight combinatorial crystals with the same lowest weight element $b_-$, and $\wt(b_-)=0$. Assume further that, for all $i \neq j \in I$ and all $b \in B$,
\begin{enumerate}

\item \label{ccc0} $e_i(b), e_i^*(b) \neq \emptyset$. 

\item \label{ccc1} $e_i^*e_j(b)= e_je_i^*( b)$.

\item \label{ccc2} $\varphi_i(b)+\varphi_i^*(b)- \langle \wt(b),
  \check\alpha_i \rangle\geq0.$ 

\item \label{ccc3} If $\varphi_i(b)+\varphi_i^*(b)- \langle \wt(b), \check\alpha_i \rangle =0$ then $e_i(b) = e_i^*(b)$.

\item \label{ccc4} If $\varphi_i(b)+\varphi_i^*(b)- \langle \wt(b), \check\alpha_i \rangle \geq 1$ then $\varphi_i^*(e_i(b))= \varphi_i^*(b)$ and  $\varphi_i(e^*_i(b))= \varphi_i(b)$.

\item \label{ccc5} If $\varphi_i(b)+\varphi_i^*(b)- \langle \wt(b), \check\alpha_i \rangle \geq 2$ then  $e_i e_i^*(b) = e_i^*e_i(b)$. 

\end{enumerate}
Then $(B, e_i, f_i) \simeq (B, e_i^*, f_i^*) \simeq B(-\infty)$.
\end{prop}

\subsection{Modulated graphs and preprojective algebras} \label{ss:MGPA-intro}

Modulated graphs (also sometimes called species) date back to work of Gabriel \cite{Gab}. The preprojective algebra construction here is due to Dlab and Ringel \cite{DR80}.

 Fix an undirected graph $\Gamma$ with no edges connecting a vertex to itself or multiple edges. Denote the set of vertices by $I$ and the set of edges by $E$. Let $A$ be the set of directed edges, which we call arrows; so there are two arrows in $A$ for each edge in $E$. Denote the arrow from $i$ to $j$ by ${}_j a_i$. 

A modulated graph $M$ is a graph $\Gamma$ as above along with a choice of a field $\mathbb{F}$ and: 
\begin{itemize}
\item A finite extension $\F_i$ of $\F$ for each vertex $i$ such that $\displaystyle \cap_i \F_i = \F$. 
\item An $(\F_j, \F_i)$ bimodule  ${}_j M_i$ for each ${}_j a_i$ where the two actions of $\F \subset \F_i, \F_j$ agree.
 
\item A non-degenerate $\F_i$-bilinear form $\epsilon_i^j: {}_i M_j \otimes_{\F_j} {}_j M_i \rightarrow \F_i$ for each ${}_j a_i \in A$.
\end{itemize}
The tensor algebra $T_M$ is 
\begin{equation} \label{eq:TM-path}
T_M = \bigoplus_{i_1 i_2 \cdots i_k \text{ a path in } \Gamma} {}_{k} M_{k-1} \otimes_{ \F_{k-1}}  \cdots   \otimes_{ \F_{i_3}}  {}_{i_2} M_{i_2} \otimes_{\F_{i_{2}}}  {}_{i_{2}} M_{i_1},
\end{equation}
with multiplication being tensor product if the end of one path agrees with the beginning of the next, and $0$ otherwise. 

For each ${}_j a_i \in A$, the bilinear form $\epsilon_i^j$ defines a canonical element $r_j^i \in {}_j M_i \otimes_{\F_i} {}_i M_j$: 
\begin{equation} \label{eq:ri}
r_j^i:= \sum_k v_k \otimes v^k
\end{equation}
for any pair of dual $\F_i$ bases $\{ v_k \} \subset {}_j M_i, \{ v^k\} \subset {}_i M_j$. It is well known that this does not depend on the choice of dual bases. Although the two $\F_j$ actions on ${}_j M_i \otimes_{\F_i} {}_i M_j$ need not agree, it is true that $zr_j^i=r_j^i z$ for all $z \in \F_j$.  

For each $i \in I$, define 
\begin{equation}r_i:= \sum_{j: {}_j a_i \in A} r_i^{j}.
\end{equation}

\begin{defn} \label{def:preproj-alg} (\cite[Introduction]{DR80})
The preprojective algebra $\Lambda_M$ is the quotient of $T_M$ by the ideal generated by $\{ r_i\}_{i \in I}$. 
\end{defn}

Associate a symmetrizable Cartan matrix $C=(c_{ij})$ to a modulated graph by
\begin{equation}
c_{ij}= 
\begin{cases}
2 \quad \text{ if } i=j \\
-\dim_{\F_i} {}_i M_j \quad \text{if there is an arrow from $i$ to $j$} \\
0 \qquad \text{otherwise}.
\end{cases}
\end{equation}
 As in \cite[Introduction]{DR80}, $\Lambda_M$ is finite dimensional over $\F$ if and only if $C$ is of finite type. 
 If $C$ is symmetric then, taking $\F_i=\C$ for all $i$ and well chosen bimodules and bilinear forms, one recovers Lusztig's preprojective algebra from \cite{L91} (see \cite{Ringel:1998}). But, even then, different bilinear forms can give non-isomorphic algebras (see \cite[\S6]{Ringel:1998} or \S\ref{ss:CCstrange} below). 
 
  \begin{example} \label{ex:C2-intro}
  Let $\Gamma$ be the graph where $I= \{1,2\}$ and $E$ consists of a single edge joining these vertices. 
Consider the modulated graph with $\F= \F_1=\R$, $\F_2=\C$, ${}_1 M_2={}_2 M_1 =\C$ with the standard actions of $\R$ and $\C$ by multiplication, and bilinear forms
\begin{equation}
\begin{aligned}
\epsilon_1^2: \C \otimes_\C \C & \rightarrow \R \\
z \otimes w & \rightarrow \Re(zw),
\end{aligned}
\qquad
\begin{aligned}
\epsilon_2^1: \C \otimes_\R \C & \rightarrow \C \\
z \otimes w & \rightarrow zw \; .
\end{aligned}
\end{equation}
The corresponding Cartan matrix is of type $C_2$.  Consider the elements of $T_M$:
\begin{itemize}
\item $e_1 = 1 \in \F_1$ and $e_2 = 1 \in \F_2$ in degree $0$. 

\item $\tau = 1 \in {}_2 M_1$ and  $\bar \tau = 1 \in {}_1 M_2$ in degree $1$.

\end{itemize}
The relations defining the preprojective algebra $\Lambda_M$ are
\begin{equation}
\overline{\tau}\tau=0 \quad \text{ and } \quad {\tau} {\overline{\tau}}-i {\tau} {\overline{\tau}}i =0.
\end{equation} 
As a real vector space, 
\begin{equation}
\Lambda_M= \mathbb{R} e_1\oplus  \mathbb{C} e_2\oplus \mathbb{C} \tau \oplus \mathbb{R} \overline \tau  \oplus \mathbb{R} \overline \tau i \oplus \mathbb{C} \tau \overline \tau 
\oplus \mathbb{R}  \overline \tau  i \tau.
\end{equation}
 \end{example}

\subsection{Nilpotent Representation varieties} \label{ss:rvintro}
Fix a modulated graph $M$ and let $\Lambda=\Lambda_M$. There is a natural partition of the identity $e \in \Lambda$ as $e=\sum_i e_i$, where $e_i$ is the lazy path at node $i$. Any $\Lambda$-module $V$ decomposes as a vector space as $V= \bigoplus_i e_i V$, and each $e_i V$ is naturally a vector space over $\F_i$. Given a dimension vector ${\bf v}= (v_i)_{i \in I}$, fix a $v_i$ dimensional vector space $V_i$ over $\F_i$ for each $i$. Define $\Lambda({\bf v})$ to be the variety of representations of $\Lambda$ on $V= \oplus_i V_i$ such that $e_i V= V_i$, the induced vector space structure on $V_i$ agrees with the original vector space structure, and all sufficiently long paths act as $0$. Then 
$\Lambda({\bf v})$ a subset of 
\begin{equation}
\bigoplus_{{}_j a_i \in A} \Hom_{\F_{j}}( {}_jM_i \otimes_{\F_i} V_i, V_j).
\end{equation}
Denote a point in this space by the collection ${}_j x_i \in\Hom_{\F_{j}}( {}_jM_i \otimes_{\F_i} V_i, V_j)$. Then $\Lambda({\bf v})$ is
cut out by the equations stating that 
\begin{itemize}
\item each $r_i$ acts as $0$, and

\item for all $i_1, i_2, \ldots i_n,$ where $n > \dim_{\F} V$ and each $i_k \rightarrow i_{k+1}$ is an arrow in $A$, the composition ${}_{i_n} x_{i_{n-1}} \circ \cdots \circ  {}_{i_2} x_{i_1} :  {}_{i_n} M_{i_{n-1}} \otimes \cdots \otimes  {}_{i_2} M_{i_1} \otimes V_{i_1} \rightarrow V_{i_n}$ is zero. 

\end{itemize}
In particular, $\Lambda({\bf v})$ is the set of $\F$ points of an algebraic variety. 
 
\begin{example} \label{2.7}
Consider the modulated graph from Example \ref{ex:C2-intro}. An element of $\Lambda((1,1))$ is defined by 
\begin{itemize}
\item ${}_2 x_1 \in \Hom_\C ( \C \otimes_\R \R, \C).$ Write this as $(z_1, z_i) \in {\Bbb R}^2$, where ${}_2 x_1(1 \otimes 1)=z_1+z_i i$.

\item  ${}_1 x_2 \in \Hom_\R ( \C \otimes_\C \C, \R)$. Write this as $(w_1, w_i) \in \R^2$, where ${}_1 x_2 (1 \otimes 1)= w_1$ and  ${}_1 x_2 (1 \otimes i)= w_i$. 
\end{itemize}
The conditions that the $r_i$ act as zero become equations as follows:
\begin{equation}
\begin{aligned} & r_1=0 \ \Leftrightarrow \    {}_1 x_2 \circ {}_2 x_1 =0  \ \Leftrightarrow \
{}_1 x_2 \circ {}_2 x_1 (1)=0 
\ \Leftrightarrow  \ w_1z_1+ w_i z_i = 0, 
\end{aligned}
\end{equation}
\begin{equation}
\begin{aligned}
&r_2=0     \Leftrightarrow     {}_2 x_1 \circ {}_1 x_2 - i {}_2 x_1 \circ {}_1 x_2 i  =0    \Leftrightarrow    
z_1w_1+ z_i w_i=0, \; z_i w_1- z_1 w_i=0.
\end{aligned}
\end{equation}
These imply nilpotency, so $\Lambda((1,1))$ is the set of ${\Bbb R}$-points of the algebraic variety cut out by these equations. 
There are two irreducible components defined by $\{w_1=w_i=0 \}$, and by $\{z_1=z_i=0\}$.
The corresponding real algebraic variety would contain a third component defined by $\{z_i^2=-z_1^2, w_1^2=-w_i^2, z_1w_1=-z_iw_i \}$. This contains no new ${\Bbb R}$-points, and if we base change to $\C$ would decompose further into two components:
 $\{ i z_1 = z_i , i w_1 = w_i  \}$ and $\{ -i z_1 =  z_i , -i w_1 = w_i  \}$. It is crucial that we do not include this component. That is, that we work with the space of ${\Bbb R}$-points, not the abstract algebraic variety. 
 \end{example}

\subsection{Topology}
The $\F$ points of an algebraic variety $X$ form a topological space with the Zariski topology: closed sets are locally defined as the zero sets of some polynomials. 
Recall that $X$ is irreducible if it is not the union of any two proper closed subsets. In that case, the dimension of $X$ is the maximal $d$ such that there is a sequence of irreducible subsets 
\begin{equation}
\emptyset \subsetneq X_0 \subsetneq \cdots \subsetneq X_d = X
\end{equation}
If $\F$ is infinite and $X$ is birationally equivalent to $\F^k$ then $\dim X =k$.
If $X$ is reducible, its irreducible components are the irreducible subsets which are not properly contained in larger irreducible subsets. 

The following is well-known and not difficult. For example, the case where $Y$ is irreducible follows from \cite{sp}, since any fiber bundle map is open. The general case is then immediate.

\begin{lemma} \label{lem:top}
If $\pi: X \rightarrow Y$ is a locally trivial fiber bundle with irreducible fiber $F$, then there is a bijection between the irreducible components of $X$ and $Y$. If $Y$ (or equivalently $X$) is irreducible, then $\dim X = \dim Y+ \dim F$.  \qed
\end{lemma}

\section{Realization of $B(-\infty)$}
 
 Fix a modulated graph $\M$ with Cartan matrix $C$ and preprojective algebra $\Lambda$. From now on assume that $|\F|=\infty$. Fix $V= \oplus_{i \in I} V_i$, where each $V_i$ is a vector space over $\F_i$.
Let $\bfv$ be the dimension vector of $V$, and define $\dim V \in Q$ by 
\begin{equation}
\dim V= \sum_{i \in I}v_i  \alpha_i.
\end{equation}
We sometimes abuse notation by e.g. using $(\bfv, \bfv)$ to mean $( \sum_{i \in I}v_i  \alpha_i,  \sum_{i \in I}v_i  \alpha_i)$. 
 
 \subsection{Some spaces and maps.} 
 
 \begin{definition}
For each $i \in I$, set $\displaystyle V^i = \bigoplus_{j : {}_j a_i \in A} {}_i M_j \otimes_{\F_{j} 
}V_{j}$. 
  \end{definition}
  A simple calculation shows 
  \begin{equation} \label{eq:dvi}
\dim_\F V^i=  d_i \dim_{\F_i} V^i = 2 d_i v_i-(\bfv, \alpha_i) = d_i (2v_i- \langle \bfv, \check \alpha_i \rangle).
  \end{equation} 
 \begin{definition}
Recall the canonical element $r_i^j$ from \eqref{eq:ri}. 
For each ${}_j a _i \in A$, define 
 $$
 \begin{aligned}
  {}_j\iota_i: V_i & \rightarrow   {}_i M_j \otimes_{\F_j}  {}_j M_i \otimes_{\F_i} V_i \\
 v &  \rightarrow r_i^j \otimes v \ \ .
 \end{aligned}$$ 
 \end{definition}

 Fix $x = \{ {}_j x_i  \} \in \Lambda({\bf v})$. 
 
 \begin{definition}
For ${}_j a _i \in A$, define 
 $${}_j \tilde x_i = ( \text{id} \otimes {}_j x_i ) \circ {}_j\iota_i : V_i \rightarrow {}_i M_j \otimes V_j.$$
 \end{definition}

 \begin{definition} \label{def:txa}
$\displaystyle \tilde x_i= \bigoplus_{j : {}_j a_i \in A}  {}_j \tilde  x_i : V_i \rightarrow V^i, \quad \text{ and } \quad \displaystyle  {}_i \tilde x = \bigoplus_{j : {}_j a_i \in A}  {}_i x_j: V^i \rightarrow V_i.$
 \end{definition}

 \begin{proposition} \label{prop:are-linear}
 The maps $\tilde x_i$ and ${}_i \tilde x$ are both $\F_i$ linear. 
 \end{proposition}
 
\begin{proof}
The map ${}_j\iota_i$ intertwines the left  $\F_i$-module structure on $V_i$ with the left $\F_i$-module structure on $ {}_i M_j$ (because $z r_i^j= r_i^j z$ for all $z \in \F_i$), which immediately implies that each ${}_j \tilde x_i$ is $\F_i$ linear, so $\tilde x_i$ is as well. 
That ${}_i \tilde x$ is $\F_i$ linear is immediate.
\end{proof}

\begin{definition}
For each $i$, let $S_i$ be the simple $\Lambda$-module such that  $e_i S_i = S_i$ and $\dim_{\F_i} e_i S_i= 1$. That is, $S_i$ is a copy of $\F_i$ lying over vertex $i$, and all ${}_j x_i$ are $0$.
\end{definition}

\begin{lemma} \label{lem:kformulas} Fix a representation $V$ of $\Lambda$. Then
$\Hom(S_i,V)$, $\mbox{} \hspace{-0.1cm} \Hom(V, S_i)$,   
$\Ext^1(S_i, V)$ and $\Ext^1(V,S_i)$ are all naturally $\F_i$ vector spaces, with
\begin{itemize}
\item $\Hom(S_i, V) \simeq \ker \tilde x_i$.

\item $\Hom(V, S_i) \simeq (V_i/\im {}_i \tilde x)^*$.
 
\item $\dim_{\F_i} \Ext^1(S_i, V) = \dim_{\F_i} \Ext^1(V, S_i) = \dim_{\F_i} V^i- \dim_{\F_i} \im \tilde x_i- \dim_{\F_i}  \im {}_i \tilde x.$ \end{itemize}
\end{lemma}
\begin{proof} The first two statements are obvious. For the third, we prove only the statement about $\Ext^1(V, S_i)$, since the other is similar. We seek to classify extensions 
\begin{equation}
0 \rightarrow S_i \xrightarrow{\iota} V' \xrightarrow{f} V \rightarrow 0
\end{equation}
up to equivalence. Clearly $V'_j=V_j$ if $j \neq i$. Choose a vector space splitting $V'_i \simeq V_i \oplus \mathbb{F}_i$, where $\mathbb{F}_i=\im \iota$. An extension is uniquely determined by an $\F_i$ linear map $\phi: V^i \rightarrow \F_i$ subject to the condition that the composition
\begin{equation}
V_i \oplus \mathbb{F}_i \xrightarrow{(\tilde{x}_i,0)} V^i \xrightarrow{(_i \tilde{x}, \phi)} V_i \oplus \mathbb{F}_i
\end{equation}
is 0. This precisely says that $\ker \phi \supset \im \tilde x_i$, so $\phi \in \Hom (V^i/ \text{im } \tilde{x}_i, \F_i)$.

Two maps $\phi$, $\phi'$ give rise to the same class in $\text{Ext}^1(S_i, V)$ if there exists a map 
\begin{equation}
\theta: V_i \oplus \F_i \rightarrow V_i \oplus \F_i
\end{equation}
which is the identity on $\F_i$ and on $(V_i \oplus \F_i)/\F_i$, and such that 
\begin{equation}
({}_i \tilde x,\phi')= \theta \circ ({}_i \tilde x, \phi). 
\end{equation}
Such maps are exactly $(v,x) \rightarrow (v, x+\kappa(v))$
for linear $\kappa: V_i \rightarrow \mathbb{F}_i$, and stabilize the short exact sequence if and only if $\ker \kappa \supset  \im {}_i \tilde{x}$. Thus the orbit of a short exact sequence is parameterized by $\kappa|_{\im {}_i \tilde x}$, and the result follows. 
\end{proof}

\begin{lemma} \label{lem:CBf} For any finite dimensional representation $V$ of $\Lambda$ and any $i \in I$, 
$$\dim \Ext^1(S_i, V)= \dim \Ext^1(V, S_i)= \dim \Hom(S_i, V)+ \dim \Hom(V, S_i) - \langle \dim V, \check\alpha_i \rangle,$$
where dimensions are all over $\F_i$.
\end{lemma}

\begin{proof} By Lemma \ref{lem:kformulas} it is enough to consider $ \text{Ext}^1(S_i, V)$, and we have
\begin{equation}
\begin{aligned} &\text{\text{dim} Ext}^1(S_i, V) - \text{\text{dim} Hom}(S_i, V) - \text{\text{dim} Hom}(V, S_i) + \langle \text{\text{dim}} V, \check\alpha_i \rangle \\  
 =  &  \text{\text{dim}} (V^i / \text{im}(\tilde{x}_i)) \hspace{-0.03cm} - \hspace{-0.03cm} \text{dim}(\text{im}(_i \tilde{x})) - \text{\text{dim}} (\text{ker}(\tilde{x}_i)) \hspace{-0.03cm} - \hspace{-0.03cm} \text{\text{dim}}(V_i /\text{im}(_i \tilde{x})) \hspace{-0.03cm} + \hspace{-0.03cm} \langle \text{\text{dim}} V, \check\alpha_i \rangle  \\ 
=  & \dim V^i  - \text{dim}(\text{im}(\tilde{x}_i))  - \text{dim}(\text{ker} (\tilde{x}_i))- \dim V_i + \langle \text{\text{dim}} V, \check\alpha_i \rangle \\ 
=  & \dim V^i - 2 \dim V_i + \langle \text{\text{dim}}V, \check \alpha_i \rangle = 0. 
\end{aligned}
\end{equation} 
The last equality uses \eqref{eq:dvi}.
\end{proof}

\begin{remark}
For Lusztig's preprojective algebra, Lemma \ref{lem:CBf} still holds if $S_i$ is replaced by an arbitrary finite dimensional module $W$ (see \cite[Lemma 1]{C-B}). However, for Dlab and Ringel's preprojective algebras, this is not true (see \S\ref{ss:CCstrange}).
\end{remark}

\subsection{Relations between components} \label{ss:rel}
For each $i \in I$, 
define $\varphi_i: \Lambda(\bfv) \mapsto {\Bbb Z}_{\geq 0}$ by 
\begin{equation}
\varphi_i(x) = \dim_{\F_i} \ker \tilde x_i.
\end{equation}
Let $\Lambda(\bfv)_{i;k}$ be the subset of $\Lambda(\bfv)$ where $\varphi_i$ takes the value $k$. For each $k$ this is an open subset of a closed subset of $\Lambda(\bfv)$, which is to say $\varphi_i$ is constructible. 

Fix ${\bfv}$ and $k$, and let $\bar {\bfv}= \bfv- k \alpha_i$. Fix vector spaces $V, \bar V$ of graded dimensions $\bfv, \bar \bfv,$ such that $\bar V_j=V_j$ for all $j \neq i$. Let $\Lambda(\bfv;i;k)$ be the variety whose points consist of an element of $\Lambda(\bfv)_{i,k}$ along with a short exact sequence 
\begin{equation}
0 \rightarrow \F_i^k  \rightarrow V  \rightarrow \bar V \rightarrow 0
\end{equation}
which is trivial on $V_j$ for all $j \neq i$. 
Explicitly, 
$\Lambda(\bfv; i; k)$ is the set of triples $(x,p,q)$ where $x \in \Lambda(V)_{i;k}$, $p: V _i \mapsto \bar V_i$ is surjective, $q:  \F_i^k \mapsto V $ is injective, and $\ker p = \im q= \ker x_i$.
Consider the obvious projections
\begin{equation} \label{eq:3pis} \Lambda(\overline v)_{i;0} \overset{\pi_3}{\longleftarrow} \Lambda(\bar v)_{i;0} \times \Surj (V_i, \bar V_i) \overset{\pi_2}{\longleftarrow} \Lambda(\bfv; i; k) \overset{\pi_1}{\longrightarrow} \Lambda(v)_{i;k}, \end{equation}
where $\Surj(V_i, \bar V_i) $ is the space of surjective linear maps. The following generalizes \cite[Lemma 5.2.3]{KS97} (see also \cite[\S 12]{L91}). 

\begin{prop} \label{prop:fb}
$\pi_1$ is a locally trivial fiber bundle with fibers isomorphic to 
$$\GL(\bar V_i) \times \GL(\F_i^k).$$ 
$\pi_2$ is a locally trivial fiber bundle with fibers isomorphic to 
$$ \GL(\F_i^k) \times \Hom ( \F_i^{\dim V^i - \dim \overline V_i}, \F_i^k).$$ 
In particular, this gives a bijection between $\Irr \; \Lambda(\overline v)_{i;0}$ and $\Irr \; \Lambda(v)_{i;k}$.
\end{prop}

\begin{remark}
The $\Hom ( \F_i^{\dim V^i - \dim \overline V_i}, \F_i^k)$ in the fiber of $\pi_2$ is because, to extend an element of $\Lambda(\bar V)$ to an element of $\Lambda(V)$, where $V= \bar V \oplus \F_i^k$, one must choose a map from $\bar V^i$ to $F_i^k$, subject to the condition that its kernel contains $\im \tilde x_i$.  
\end{remark}

\begin{proof}[Proof of Proposition \ref{prop:fb}]
Let $r=\dim V^i- \dim \bar V_i$. Fix the following data:
\begin{itemize}

\item An isomorphism $\psi: \overline{V}_i \oplus \mathbb{F}_i^k \rightarrow V_i$. Let $W = \psi (\overline{V}_i) \subset V_i$, and let $\psi_1: \overline{V}_i \rightarrow W$ be the resulting isomorphism. Consider also the isomorphism $m: \F_i^k \rightarrow V_i/W$ and the projection map $\pi: V_i \rightarrow V_i/W$.  

\item An injective map  $\iota: \F_i^r \hookrightarrow V^i$.
\end{itemize}
For all $x$ in some open dense subset of $\Lambda(\bfv; i; k)$, $V_i= W \oplus \ker x_i$. Let  
\begin{equation} \pi_{\ker x_i}: V_i \rightarrow \ker x_i, \;\; \text{ and } \;\; \pi_W: V_i \rightarrow W
\end{equation}
be the corresponding projections. 
Notice that $\pi_{\ker x_i}$ is well defined on $V_i/W$, and 
\begin{equation}
\pi_{\ker x_i} \circ \pi |_{\ker x_i} = \Id, \;\; \;  \pi \circ \pi_{\ker x_i} = \Id_{V_i/W}.
\end{equation}

First consider $\pi_1$. The map
\begin{equation}
(x, p, q) \rightarrow  (x, (p \circ \psi_1, m^{-1} \circ  \pi \circ q))
\end{equation} 
is the required local isomorphism on the open set where $W \cap \ker x_i=0$, with inverse 
\begin{equation}
(x,(a, b)) \rightarrow (x, a\circ \psi_1^{-1} \circ  \pi_W, \pi_{\ker x_i} \circ m \circ b).
\end{equation}

Now consider $\pi_2$. The local isomorphisms are:
\begin{equation}
\begin{aligned}
& (x, p, q) \rightarrow \left( \left(\bar x,  p \right), \left(m^{-1} \circ \pi \circ q,  m^{-1} \circ  \pi  \circ  {}_{i} x \circ \iota\right) \right) \\
& ((\bar x, p), (r, \gamma)) \rightarrow (x, p, \pi_{\ker x_i} \circ m \circ \; r),
\end{aligned}
\end{equation}
where $x$ is the extension of $ \bar x$ to $\bar V \oplus  \F_i^k$ defined by $\gamma$, thought of as an element of $\Lambda(V)$ using the isomorphism $\psi$. 
These are defined on the open subset of  $\Lambda(\bar v)_{i;0} \times \Surj(V_i, \bar V_i)$ where $\text{ker}(p) \oplus W = V_i$ and $\text{im}(\iota) \cap \text{im}(\tilde{x}_i) = 0$. 

Since $|\F|=\infty$ the fibers are irreducible and so Lemma \ref{lem:top} gives the required bijection.
\end{proof}

\begin{definition}
Let $\displaystyle D(\bfv) =  \sum_{i \in I} d_i v_i^2.$
\end{definition}

\begin{lem} \label{lem:dbc}
$\Lambda(\bfv)$ has pure dimension $D(\bfv)-\frac{1}{2}( \bfv, \bfv )$ over $\F$. Each $\Lambda(\bfv)_{i;k}$ is also pure dimensional of this dimension, and there is a bijection $\Irr \;\Lambda(\bfv) \mapsto \coprod_k \Irr \;  \Lambda(\bfv)_{i;k}$ which takes $X$ to $X \cap \Lambda(\bfv)_{i;k}$ for the unique $k$ for which this is dense in $X$. 
\end{lem}

\begin{proof}
Proceed by induction on $\bfv$, the case $\bfv=0$ being trivial. Fix $\bfv \neq 0$, and assume the statement for all smaller $\bfv'$. 

Fix $k\geq1$.
By Proposition \ref{prop:fb},  $\Lambda(\bfv;i;k)$ is a fiber bundle over each of $\Lambda(\bfv)_{i;k}$ and $ \Lambda(\bfv - k {\bf 1}_i)_{i;0}$. By considering the dimensions of the fibers,
\begin{equation}
  \dim_\F \Lambda(\bfv)_{i;k}  =   \dim_\F \Lambda(\bfv - k \alpha_i)_{i;0} +  d_i k \dim_{\F_i} V^i. 
\end{equation}
Using induction and substituting the dimension of $V^i$ (see \eqref{eq:dvi}) gives 
\begin{equation}
\begin{aligned}
& \quad   \dim_\F \Lambda(\bfv)_{i;k}   \\ & =   D(\bfv-k \alpha_i) - \frac{1}{2}( \bfv-k \alpha_i, \bfv-k\alpha_i ) +  2 kd_i v_i-k(\bfv, \alpha_i) \\ 
  &= \{D(\bfv) - 2 d_i v_i k + d_i k^2\} \hspace{-0.03cm} - \hspace{-0.03cm}  \frac{1}{2}\{ (\bfv, \bfv) \hspace{-0.03cm} - \hspace{-0.03cm}  2k (\bfv, \alpha_i) + 2 d_i k^2 \} +  2 kd_i v_i-k(\bfv, \alpha_i) \\ &= D(\bfv) - \frac{1}{2}( \bfv ,\bfv ).
  \end{aligned} 
\end{equation}

Now, fix $X \in  \Irr \Lambda(\bfv)$. Every point is a nilpotent representation, so is in $\Lambda(\bfv)_{i; k}$ for some $i$ and $k \geq 1$. In particular, $X$ has an open dense subset contained in $\Lambda(\bfv)_{i;k}$ for some $i$ and $k\geq 1$. The result follows for $X$ since it is true for $\Lambda(\bfv)_{i;k}$. 

Finally we must handle the case of $\Lambda(\bfv)_{i;0}$. But this is open in $\Lambda(\bfv),$ so every irreducible component is open and dense in some irreducible component of $\Lambda(\bfv)$, and the result follows by the previous paragraph.

The required bijections of components are clear. 
\end{proof}

\subsection{Crystal operators} \label{ss:crystalops}
Let
\begin{equation}
B= \coprod_\bfv \Irr \Lambda(\bfv).
\end{equation}
By Proposition \ref{prop:fb}, we have bijections
\begin{equation}
f_{i;k}:  \coprod_\bfv \Irr \; \Lambda(\bfv)_{i;k} \rightarrow \coprod_\bfv  \Irr\;  \Lambda(\bfv)_{i;0}.
\end{equation}
Define 
\begin{equation}
 f_i := \bigsqcup_k  f_{i;k-1}^{-1}   f_{i;k},
\quad\quad 
 e_i := \bigsqcup_k f_{i;k+1}^{-1} f_{i;k},
\end{equation}
where for $X \in \Irr \; \Lambda(\bfv)_{i;0}$ we set $f_i(X) = \emptyset$. 
By Lemma \ref{lem:dbc}, $\Irr \; \Lambda(\bfv)$ is in bijection with $\coprod_k \Irr \; \Lambda(\bfv)_{i;k}$, where $X$ corresponds to the component of $\coprod_k \Irr \; \Lambda(\bfv)_{i;k}$ that is dense in $X$. This gives operators $ f_i$ and $ e_i$ on $B$.

We also need the $*$ operators, which are constructed in an analogous way. 
Define
\begin{equation}
 \varphi_i^*(x) = \dim_{\F_i} V_i/\text{im}({}_i \tilde x) \quad \text{and} \quad \Lambda(V)_{i}^k  = \{ x \in \Lambda(v): \varphi_i^*(x)=k \}.
 \end{equation}
 Let $ \Lambda^*(\bfv; i; k)$ be the variety whose points consist of on element of $\Lambda(V)_{i}^k$ along with a short exact sequence
 \begin{equation}
 0 \rightarrow \bar V \rightarrow V \rightarrow {\Bbb F}_i^k \rightarrow 0.
 \end{equation}
Consider the natural projections 
$\pi^*_1 : \Lambda^*(\bfv;i;k) \rightarrow \Lambda(\bfv)_{i}^k$ and 
$\pi^*_2 : \Lambda(\bfv;i;k) \rightarrow \Lambda(\bar \bfv)_i^0.$
$\pi_2^* \circ (\pi_1^*)^{-1}$ defines a bijection $ f_{i;k}^*$ on the level of irreducible components.
Let  
\begin{equation}
\f_i^* =\coprod_k (\f_{i; k-1}^*)^{-1} \circ \f_{i;k}^* \quad \text{and } \quad  e_i^* = \coprod_k ( f_{i; k+1}^*)^{-1} \circ \f_{i;k}^*.
\end{equation}
$\Irr \Lambda(\bfv)$ and $\coprod_k \Irr \Lambda(\bfv)_i^k$ are in bijection, so this gives operators on  $B.$

\subsection{Reworded operators}

Recall that a property is said to hold generically on a topological space if it holds on an open-dense subset. 

\begin{prop} \label{prop:gen-def}
Fix $X \in \Irr \Lambda(\bfv)$. For generic $x \in X$ and a generic extension
$$
0 \rightarrow S_i \rightarrow (V', x') \rightarrow (V, x) \rightarrow 0,$$
$(V', x')$ is in a single irreducible component $Y \in \Irr \Lambda(V')$, and $Y= \e_i X$.
Furthermore, the subset of $Y$ which can be realized in this way is open-dense. 

Similarly, for generic $x \in X$ and a generic extension
$$
0 \rightarrow (V, x)  \rightarrow (V', x') \rightarrow S_i \rightarrow 0,$$
$(V', x')$ is in a single irreducible component $Y \in \Irr \Lambda(V')$, and $Y= \e_i^* X$.
Furthermore, the subset of $Y$ which can be realized in this way is open-dense. 
\end{prop}

\begin{proof}
Consider the first statement. 
 Let $k$ be the generic value of $\varphi_i$ on $X$. Let $\widehat X_o$ be the subset of $X \cap \Lambda(\bfv)_{i; k}$ consisting of points which are not in any other irreducible component of $\Lambda(\bfv)$. Clearly $X \backslash \widehat X_o$ is closed and $\widehat X_o$ is non-empty, so $\widehat X_o$ is open dense. 
 Consider  
\begin{equation} \label{eq:4pis}
\mbox{} \hspace{-0.4cm}
\xymatrix{
  & \ar@/ /[dl]_{\pi_1}  \Lambda(\bfv;k)  \ar@/ /[rd]^{\pi_3  \pi_2} && \ar@/ /[dl]_{\pi'_3  \pi'_2}  \Lambda(\bfv+\alpha_i;k+1)  \ar@/ /[rd]^{\pi'_1} & 
 \\
\Lambda(\bfv)_{i;k} && \Lambda(\bfv- k \alpha_i)_{i;0} && \mbox{} \hspace{-1cm} \Lambda(\bfv+\alpha_i)_{i;k+1}.
\\
}
\end{equation}
Recall that all these maps give bijections of irreducible components.
Let $X'= e_i(X)$, and define $\widehat X'_o$ analogously to $\widehat X_o$.
Then the stated condition holds for any $x$ in
\begin{equation}
X_o =  \widehat X_o \cap  \pi_1 (\pi_3 \pi_2)^{-1} \pi'_3\pi'_2 (\pi'_1)^{-1} \widehat X'_o,
\end{equation}
which is open dense.
The second statement follows by a symmetric argument. 
\end{proof}

\subsection{The realization}
Recall that $B = \coprod_\bfv \Irr \; \Lambda(\bfv)$. For each $X \in \Irr \Lambda(\bfv)$, define:
\begin{itemize}
\item $\wt(X)=\sum_i v_i \alpha_i$.

\item $\displaystyle \varphi_i(X)= \min_{x \in X} \varphi_i(x) =  \min_{x \in X} \dim_{\F_i} \ker \tilde x_i, \quad  \varepsilon_i(X)= \varphi_i(X) - \langle \wt(X), \check\alpha_i \rangle.$

\item $\displaystyle \varphi_i^*(X)= \min_{x \in X} \varphi^*_i(x) =   \min_{x \in X} \dim_{\F_i} V_i/\text{im}({}_i \tilde x), \quad \varepsilon^*_i(X)= \varphi^*_i(X)-\langle \wt(X), \check\alpha_i \rangle.$

\end{itemize} 

\begin{lemma} \label{lem:pcc} $B$ along with either $\{ e_i, f_i\}$ or $\{ e_i^*, f_i^* \}$ and the additional data defined above is a lowest weight combinatorial crystal, where the lowest weight element is $\Lambda({\bf 0}).$ 
\end{lemma}

\begin{proof}
For either structure the conditions in Definitions \ref{def:crystal} are immediate from the construction. The condition that any $x \in \Lambda(\bfv)$ is nilpotent implies that, for any $X \in B$ of weight $\neq 0$, there are $i$ and $j$ such that $\f_i X, \f_j^* X \neq 0.$ This, along with the definition of $\varphi_i(X), \varphi_i^*(X)$, shows that these combinatorial crystals are lowest weight.   
\end{proof}

\begin{lemma} \label{lem:tcb}
Fix $X \in B$. For generic $T \in X$, 
 $$\varphi_i(X)+\varphi_i^*(X)- \langle \wt(X), \check\alpha_i \rangle = \dim \Ext^1(T,S_i)= \dim \Ext^1(S_i,T).$$
\end{lemma} 

\begin{proof}
This is immediate from Lemma \ref{lem:CBf} and the definitions of $\varphi_i(X)$ and $\varphi_i^*(X). $
\end{proof}

\begin{prop} \label{prop:scc}
Fix $X \in B$ and $i,j \in I.$ 
\begin{enumerate}

\item \label{scc1} If $\varphi_i(b)+\varphi_i^*(b)- \langle \wt(b), \check\alpha_i \rangle = 0$, then $e_i(X)=e_i^*(X)$.

\item \label{scc2} If $\varphi_i(b)+\varphi_i^*(b)- \langle \wt(b), \check\alpha_i \rangle \geq 1$, then $\varphi_i (e_i^*(X))=\varphi_i(X)$ and $\varphi^*_i (e_i(X))=\varphi^*_i(X)$. 

\item \label{scc3} If either $i \neq j$ or $\varphi_i(b)+\varphi_i^*(b)- \langle \wt(b), \check\alpha_i \rangle \geq 2$, then $e_i^*e_j(X) = e_j e_i^*(X)$. 

\end{enumerate}
\end{prop}

\begin{proof}
Fix $X$ and let $T$ be the representation corresponding to a generic point in $X$, meaning one where all $\varphi_i, \varphi_i^*$ are minimal. 

In case \eqref{scc1},  by Lemma \ref{lem:tcb}, $\Ext^1(T,S_i)= \Ext^1(S_i,T)=0$, so the generic extensions in Proposition \ref{prop:gen-def} are in fact trivial extensions, and $e_i(X)=e_i^*(X)$.

In case \eqref{scc2}, by Lemma \ref{lem:tcb}, $\dim \Ext^1(T,S_i)>0$, so, if $T'$ is the generic extension 
\begin{equation}
0 \rightarrow S_i \rightarrow T' \rightarrow T \rightarrow 0,
\end{equation}
 from Proposition \ref{prop:gen-def}, then $ \Hom(T', S_i) \simeq \dim \Hom (T, S_i)$. Using Lemma \ref{lem:kformulas}, $\varphi_i (e_i^*(X))=\varphi_i(X)$. The other equality is true by a similar argument. 

In case \eqref{scc3}, consider generic $T' \in e_j X$ and $T''$ in $e_i^* e_j X$. We claim that 
the natural homomorphism from the $i$-socle of $T''$ to the $j$-head is trivial. 
If $i \neq j$ this is clear. If $i=j$, it suffices to show that $S_i$ is not a direct summand of $T''$. First, notice that $S_i$ cannot be a summand of $T$: If $T= S_i+\overline T$ then, since $T$ is generic, this would imply $\text{Ext}^1(\overline{T}, S_i) = 0$, and hence $\text{Ext}^1(T, S_i) = 0$, which is false by Lemma \ref{lem:tcb}. Since $\Ext(S_i, T) >0$, using Proposition \ref{prop:gen-def}, a generic $T' \in e_i X$ also doesn't contain $S_i$ as a direct summand. But then, using Lemma \ref{lem:CBf},
\begin{equation}
\text{dim Ext}^1(T',S_i) 
= \text{dim Ext}^1(T, S_i) - 1 > 0.
\end{equation}
The same argument then shows that $T''$ also does not contain $S_i$ as a direct summand. 

By Proposition \ref{prop:gen-def}, applying either $f_i^* f_j$ or $f_j f_i^*$ generically takes a subquotient that decreases the dimension of both the $i$-head and $j$-socle by 1. The homomorphism from $i$-head to $j$-socle is trivial so these operations commute. Hence 
\begin{equation}
f_i^* f_j e_i^*e_j X= f_jf_i^* e_i^* e_j X=X.
\end{equation}
But $ f_i^* f_j e_j e_i^* X=X$ as well so, using Definition \ref{def:crystal}\eqref{def:dcp3}, $e_i^*e_j X=  e_j e_i^* X$. 
\end{proof}

\begin{theorem} \label{th:iscrystal}
$B$ is a realization of $B(-\infty)$.  
\end{theorem}

\begin{proof}
By Proposition \ref{lem:pcc} it remains to check the conditions of Proposition \ref{prop:comb-characterizaton2}. Condition 
\eqref{ccc0} is immediate, \eqref{ccc2} is clear from Lemma \ref{lem:tcb}, and \eqref{ccc1},\eqref{ccc3},\eqref{ccc4},\eqref{ccc5} are Proposition \ref{prop:scc}. 
\end{proof}

\section{Examples} 

\subsection{Type $C_2$  }
Consider the modulated 
graph from Example \ref{ex:C2-intro}, where $\F_1=\R$ and $\F_2=\C$. In this case $\Lambda$ is representation-finite, and each indecomposable representation can be uniquely identified by its socle filtration, which we record from right to left. So, for example,  $\C\R^2$ means the unique indecomposable with a copy of the simple $\C$ over vertex $2$ in its head, and two copies of the simple $\R$ over vertex $1$ in its socle. 
The isomorphism classes of indecomposable $\Lambda$-modules are listed in Figure \ref{fig:indecomposables}. 

\begin{figure}

\hspace{0.5cm} \xymatrix{
& \R \ar@{-}[rr] && \R\C && \R\C\R \\
\C \ar@{-}[ru] \ar@/^/@{-}[rrrr] && \C\R \ar@{-}[lu] \ar@{-}[ru]&& \C\R^2  \ar@{-}[lu]&& \C\R^2\C \\
&& \R^2\C \ar@{-}[llu]  \ar@{-}[u]  \ar@{-}[rru]
}

\caption{\label{fig:indecomposables} Indecomposable representations for the preprojective algebra of type $C_2$. Lines connect pairs which admit a non-trivial extension.}
\end{figure}
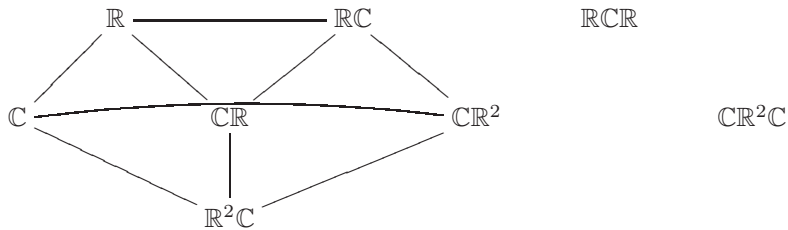

To see that there are no other indecomposables, first check that $\mathbb{R} \mathbb{C} \mathbb{R}$ and $\mathbb{C} \mathbb{R}^2 \mathbb{C}$ are the indecomposable projectives. Next fix some module $V$. If an element of $ {}_2 M_1 \otimes {}_1 M_2$ acts non-trivially on some $v \in V_2$, it is easy to see that the sub-module generated by $v$ is isomorphic to $\mathbb{C} \mathbb{R}^2 \mathbb{C}$, and hence occurs as a direct summand. Similarly, if 
an element of $ {}_1 M_2 \otimes {}_2 M_1$ acts non-trivially on some $v \in V_1$, then the submodule generated by $v$ is isomorphic to $\mathbb{R} \mathbb{C} \mathbb{R}$, so is a summand. If there are no such elements then $V$ is a direct sum $V_1 \oplus V_2$, where $\tau|_{V_1} = 0$, and $\overline{\tau}|_{V_2} = 0$, and these are easy to analyze. 

In each irreducible component of $\Lambda(\bfv)$, the isomorphism class of the corresponding representation is constant on an open-dense set. The classes that show up this way are exactly the rigid representations: ones where no two indecomposables in the Krull-Schmidt decomposition admit a non-trivial extension. 
Hence the irreducible components of $\Lambda(\bfv)$ correspond to collections of indecomposables none of which are connected by lines in Figure \ref{fig:indecomposables}, whose total dimension in $\bfv$. For example, the number of Kostant partitions of $3 \alpha_1+2\alpha_2$ is 5, so the results above imply there must be 5 irreducible components of $\Lambda((3,2))$. The corresponding rigid modules are:  
\begin{equation}
\R \oplus \C\R^2\C, \ \ \C\R \oplus \C\R^2, \ \ \C\R \oplus \R\C\R, \ \  \R\C \oplus \R^2\C, \ \  \R\C \oplus \R\C\R.
\end{equation}
Unfortunately, $\Lambda$ usually has infinitely many isomorphism classes of indecomposable representations, even in finite type, so this method does not generalize.

\subsection{Deformed pre-projective algebra over ${\Bbb C}$ for $\widehat{\mathfrak{sl}}_2$} \label{ss:CCstrange}

Dlab and Ringel's construction need not agree with Lusztig's even when all the $\F_i$ are chosen to be $\C$. 
Such an example is given in \cite[\S 6]{Ringel:1998} for $\widehat{\mathfrak{sl}}_n$, $n \geq 3$. Here we give an example for $\widehat{\mathfrak{sl}}_2$.

Consider the graph where $I=\{1,2\}$ and $E$ consists of a single edge joining 1 and 2. Choose $\F_1, \F_2=\C$, and ${}_1 M_2= {}_2 M_1 = \C^2$, with the actions of both $\F_1$ and $\F_2$ being scalar multiplication on both bimodules. The corresponding Cartan matrix is of type $\widehat{\mathfrak{sl}}_2$. Define
\begin{equation}
 \begin{aligned}
 \epsilon_1^2:  &&  {}_1 M_2   \otimes_\C {}_2 M_1 & \rightarrow \C \\
&& \mbox{} \hspace{-0.2cm} (s_1, s_2)   \otimes (t_1, t_2) & \rightarrow s_1t_1+s_2t_2
 \end{aligned}
\ \
\text{and}  
 \ \
 \begin{aligned}
\epsilon_2^1:  &&  {}_2 M_1 \otimes_\C {}_1 M_2 & \rightarrow \C \\
&&  \mbox{} \hspace{-0.2cm}  (t_1, t_2) \otimes (s_1, s_2) & \rightarrow t_1s_1-t_2s_2.
 \end{aligned}
 \end{equation}
 
Fix a graded vector space $V=V_1 \oplus V_2$. For any $x \in \Lambda(V)$, consider the four maps
\begin{itemize}
\item $m_1 = {}_2 x_1((1,0) \otimes \cdot), \; m_2 = {}_2 x_1((0,1) \otimes \cdot)$ in $\Hom(V_1,V_2)$,

\item $\overline m_1 = {}_1 x_2((1,0) \otimes \cdot), \; \overline m_2 = {}_1 x_2((0,1) \otimes \cdot)$ in $\Hom(V_2,V_1)$. 

\end{itemize}
These determine $x$, and the preprojective relations are
\begin{equation} \label{eq:pppo}
 \om_1 m_1 -  \overline m_2m_2=0  \quad \text{and} \;\;\; m_1 \om_1 + m_2\om_2=0.
 \end{equation}
Take $V_1$ and $V_2$ to both be one dimensional, with bases $\{ v_1 \}, \{ v_2 \}$ respectively. Define 
\begin{equation}
\text{ $m_1=m_2$ is  the map which sends $v_1$ to $v_2$, } \;\; \text{ and } \;\; \text{ $\om_1=\om_2 = 0.$}
\end{equation}
These satisfy \eqref{eq:pppo}, so define a module for $\Lambda$.

Now take a second copy $V'$ of this module with basis vectors $v'_1, v'_2$. Any extension of $V$ by $V'$ is determined by $a,b,c,d \in \C$ defined by
\begin{equation}
m_1(v'_1) = v'_2+ av_2, \;\;\;  m_2(v'_1)=v'_2+ bv_2 , \;\;\;  \om_1(v'_2)=cv_1,  \; \text{ and } \; \; \om_2(v'_2)=dv_1.
\end{equation}
The preprojective relations give 
\begin{equation}
c-d=0 \quad \text{and} \quad c+d=0.
\end{equation}
The only solution is $c=d=0$, so any such extension has a two-dimensional head. 

A simple calculation shows that, for Lusztig's preprojective algebra of type $\widehat{\mathfrak{sl}}_2$, any indecomposable $V$ that fits into a short exact sequence
$0 \rightarrow S_2 \rightarrow V \rightarrow S_1 \rightarrow 0$
has a self-extension with a 1-dimensional head. 
Hence, with the above choices, Dlab and Ringel's preprojective algebra is not isomorphic to Lusztig's. It can still be used in our realization of $B(-\infty)$, so our results generalize existing literature even in symmetric types. 

Note also that Lemma \ref{lem:CBf} fails here if both modules are taken to be $V$, since
\begin{equation}
\text{\text{dim} Ext}^1(V, V) - 2 \text{\text{dim} Hom}(V, V)  + (\dim V, \dim V )
= 1-2+0 \neq 0.
\end{equation}

There is actually a family of preprojective algebras parameterized by $z \in {\Bbb C}^\times$ defined as above but with 
$\epsilon_2^1 = z t_1s_1-t_2s_2.$
If $z= -1$ this is Lusztig's preprojective algebra, but for all $z \neq -1$ the above argument shows that it is not. Hence Dlab and Ringel's construction can be thought of as non-trivially deforming Lusztig's preprojective algebra in this case. 

\section{Further directions} \label{sec:directions}
There are many possible future directions for this work. Essentially, for every result proven using or about Lusztig's quiver varieties, one can ask if it can be extended to our generality. Here we briefly discuss a few examples.

\subsection{Semi-canonical and canonical bases in symmetrizable types?}
Here we only consider the crystal $B(\infty)$, but in symmetric types one can also realize $U^-(\g)$ using a convolution product on constructible functions on Lusztig's quiver varieties. Can one extend this to our quiver varieties, and hence define semi-canonical bases in full generality? Even more ambitiously, can one use some version of perverse sheaves built from modulated quivers to study canonical bases in non-symmetric types? There seem to be some obstacles. One approach would be to modify our construction to work over algebraically closed fields, since many geometric techniques work better in that setting, but this would require some new ideas (for instance, Example \ref{2.7} shows that naively base-changing to ${\Bbb C}$ would give the wrong number of irreducible components).  Geiss-Leclerc-Schr\"oer \cite{GLS1, GLS2,GLS3} have another approach to non-symmetric preprojective algebras which works over algebraically closed fields, so perhaps the first step would be to better understand how our work is related to their construction. 

\subsection{Nakajima quiver varieties, and $B(\lambda)$ crystals.} In \cite{Sai02}, Saito uses  irreducible components of Nakajima's quiver varieties to realize the integrable highest weight crystals $B(\lambda)$ in symmetric types. It should be possible to extend this to non-symmetric types by extending our construction to include Nakajima's varieties. 

\subsection{Comparing with combinatorial realizations}
In types $A$ and $D,$ Savage \cite{Sav2} describes the relationship between certain combinatorial realizations of crystals and Kashiwara-Saito's geometric realization. He also considers some simply laced affine types. It would be interesting to extend this to types $B$ and $C$, as well as to non-symmetric affine types. 

\section*{Acknowledgements}

These ideas arose from discussions with Gordana Todorov, beginning at the 2010 Maurice Auslander Distinguished Lectures
and International Conference. We thank the organizers for that wonderful event, and Dr. Todorov for her contributions. We also thank Pavel Etingof, George Lusztig and Peter J McNamara for interesting comments. The second author was partially supported by NSF grants DMS-0902649 and DMS-1265555.

\end{document}